\documentclass[12pt,twoside]{amsart}
\usepackage{amsmath, amsthm, amscd, amsfonts, amssymb, graphicx}
\usepackage[bookmarksnumbered, plainpages]{hyperref}

\newtheorem{thm}{Theorem}[section]
\newtheorem{cor}[thm]{Corollary}

\newtheorem{defn}[thm]{Definition}

\numberwithin{equation}{section}

\begin{document}

\title{\bf  $SL(2,{\bf Z})$ modular forms and anomaly cancellation formulas}
\author{Yong Wang* \ \ Jianyun Guan}

\thanks{{\scriptsize
\hskip -0.4 true cm \textit{2010 Mathematics Subject Classification:}
58C20; 57R20; 53C80.
\newline \textit{Key words and phrases:} $SL(2,{\bf Z})$ modular forms; anomaly cancellation formulas; divisibility of index
\newline \textit{* Corresponding author.}}}

\maketitle

\begin{abstract}
 By some $SL(2,{\bf Z})$ modular forms introduced in \cite{Li2} and \cite{CHZ} , we get some interesting
anomaly cancellation formulas. As corollaries, we get some divisibility results of index of twisted Dirac operators.
\end{abstract}

\vskip 0.2 true cm


\pagestyle{myheadings}
\markboth{\rightline {\scriptsize Yong Wang}}
         {\leftline{\scriptsize $SL(2,{\bf Z})$ modular forms and anomaly cancellation formulas}}

\bigskip
\bigskip


\section{ Introduction}
\quad In 1983, the physicists Alvarez-Gaum\'{e} and Witten \cite{AW}
  discovered the "miraculous cancellation" formula for gravitational
  anomaly which reveals a beautiful relation between the top
  components of the Hirzebruch $\widehat{L}$-form and
  $\widehat{A}$-form of a $12$-dimensional smooth Riemannian
  manifold. Kefeng Liu \cite{Li1} established higher dimensional "miraculous cancellation"
  formulas for $(8k+4)$-dimensional Riemannian manifolds by
  developing modular invariance properties of characteristic forms.
  These formulas could be used to deduce some divisibility results. In
  \cite{HZ1}, \cite{HZ2}, \cite{CH}, some more general cancellation formulas that involve a
  complex line bundle and their applications were established. In \cite{HLZ1}, using the Eisenstein series, a more general cancellation
  formula was derived.  In \cite{HLZ2}, Han, Liu and Zhang showed that both of the Green-Schwarz anomaly factorization formula
for the gauge group $E_8\times E_8$ and the Horava-Witten anomaly factorization formula for the gauge
group $E_8$ could be derived through modular forms of weight $14$. This answered a question of J.
H. Schwarz. They also established generalizations of these factorization formulas and obtaind a new
Horava-Witten type factorization formula on $12$-dimensional manifolds. In \cite{HHLZ}, Han, Huang, Liu and Zhang introduced a modular form of weight $14$ over $SL(2,{\bf Z})$ and a modular form of weight $10$ over $SL(2,{\bf Z})$ and they got some interesting
anomaly cancellation formulas on $12$-dimensional manifolds. In \cite{Li2}, a modular form of weight $2k$ for a $2k$-dimensional spin manifold was introduced.
  In \cite{CHZ}, Chen, Han and Zhang defined an integral modular form of weight $2k$ for a $4k$-dimensional $spin^c$ manifold and an integral modular form of weight $2k$ for a $4k+2$-dimensional $spin^c$ manifold. {\bf Our motivation is to prove some more anomaly cancellation formulas on some dimensional manifolds by modular forms over $SL(2,{\bf Z})$ introduced in \cite{Li2} and \cite{CHZ}.}\\
\indent Let $X$ be a spin manifold.
 Let $\nabla^{ TX}$ be the associated Levi-Civita connection on $TX$
 and $R^{TX}=(\nabla^{TX})^2$ be the curvature of $\nabla^{ TX}$.
 Let $\widehat{A}(TX,\nabla^{ TX})$
 be the Hirzebruch characteristic forms defined  by (cf.
 \cite{Zh})
 $$\widehat{A}(TX,\nabla^{ TX})={\rm
 det}^{\frac{1}{2}}\left(\frac{\frac{\sqrt{-1}}{4\pi}R^{TX}}{{\rm
 sinh}(\frac{\sqrt{-1}}{4\pi}R^{TX})}\right).$$
Let $\triangle(X)$ be the spinor bundle and $\widetilde{TX}=TX-{\rm dim}X$. We have
\begin{thm}
When ${\rm dim}X=8$, we have
\begin{align}
&\left\{\widehat{A}(TX){\rm ch}(\triangle(X)){\rm ch}(2\widetilde{T_CX})+32\widehat{A}(TX){\rm ch}(\widetilde{T_CX}+\wedge^2\widetilde{T_CX})\right\}^{(8)}\\\notag
&=240\left\{\widehat{A}(TX){\rm ch}(\triangle(X))+32\widehat{A}(TX)\right\}^{(8)},
\end{align}
\begin{align}
&\left\{\widehat{A}(TX){\rm ch}(\triangle(X)){\rm ch}(2\widetilde{T_CX}+\wedge^2\widetilde{T_CX}+\widetilde{T_CX}\otimes\widetilde{T_CX}
+S^2\widetilde{T_CX})\right.\\\notag
&\left.+32\widehat{A}(TX){\rm ch}(\wedge^4\widetilde{T_CX}+\wedge^2\widetilde{T_CX}\otimes\widetilde{T_CX}+\widetilde{T_CX}
\otimes\widetilde{T_CX})
+S^2\widetilde{T_CX}+\widetilde{T_CX}\right\}^{(8)}\\\notag
&=2160\left\{\widehat{A}(TX){\rm ch}(\triangle(X))+32\widehat{A}(TX)\right\}^{(8)}.
\end{align}
\end{thm}
\begin{cor}
Lex $X$ be an $8$-dimensional spin manifold without boundary, then
\begin{equation}
{\rm Ind}((D \otimes \triangle(X) \otimes \widetilde{T_CX})_+)\equiv 0 ~~{\rm mod} ~~8Z,
\end{equation}
\begin{equation}
 {\rm Ind}((D \otimes \triangle(X) \otimes (2\widetilde{T_CX}+\wedge^2\widetilde{T_CX}+\widetilde{T_CX}\otimes\widetilde{T_CX}
+S^2\widetilde{T_CX}))_+)\equiv 0 ~~{\rm mod} ~~16Z.
\end{equation}
\end{cor}

\begin{thm}
When ${\rm dim}X=12$, we have
\begin{align}
&\left\{\widehat{A}(TX){\rm ch}(\triangle(X)){\rm ch}(2\widetilde{T_CX})+128\widehat{A}(TX){\rm ch}(\widetilde{T_CX}+\wedge^2\widetilde{T_CX})\right\}^{(12)}\\\notag
&=-504\left\{\widehat{A}(TX){\rm ch}(\triangle(X))+128\widehat{A}(TX)\right\}^{(12)},
\end{align}
\begin{align}
&\left\{\widehat{A}(TX){\rm ch}(\triangle(X)){\rm ch}(2\widetilde{T_CX}+\wedge^2\widetilde{T_CX}+\widetilde{T_CX}\otimes\widetilde{T_CX}
+S^2\widetilde{T_CX})\right.\\\notag
&\left.+128\widehat{A}(TX){\rm ch}(\wedge^4\widetilde{T_CX}+\wedge^2\widetilde{T_CX}\otimes\widetilde{T_CX}+\widetilde{T_CX}
\otimes\widetilde{T_CX})
+S^2\widetilde{T_CX}+\widetilde{T_CX}\right\}^{(12)}\\\notag
&=-16632\left\{\widehat{A}(TX){\rm ch}(\triangle(X))+128\widehat{A}(TX)\right\}^{(12)}.
\end{align}
\end{thm}
\begin{cor}
Lex $X$ be an $12$-dimensional spin manifold without boundary, then
\begin{equation}
{\rm Ind}((D \otimes \triangle(X) \otimes \widetilde{T_CX})_+)\equiv 0 ~~{\rm mod} ~~4Z,
\end{equation}
\begin{equation}
 {\rm Ind}((D \otimes \triangle(X) \otimes (2\widetilde{T_CX}+\wedge^2\widetilde{T_CX}+\widetilde{T_CX}\otimes\widetilde{T_CX}
+S^2\widetilde{T_CX}))_+)\equiv 0 ~~{\rm mod} ~~8Z.
\end{equation}
\end{cor}

\begin{thm}
When ${\rm dim}X=16$, we have
\begin{align}
&\left\{\widehat{A}(TX){\rm ch}(\triangle(X)){\rm ch}(2\widetilde{T_CX})+512\widehat{A}(TX){\rm ch}(\widetilde{T_CX}+\wedge^2\widetilde{T_CX})\right\}^{(16)}\\\notag
&=480\left\{\widehat{A}(TX){\rm ch}(\triangle(X))+512\widehat{A}(TX)\right\}^{(16)},
\end{align}
\begin{align}
&\left\{\widehat{A}(TX){\rm ch}(\triangle(X)){\rm ch}(2\widetilde{T_CX}+\wedge^2\widetilde{T_CX}+\widetilde{T_CX}\otimes\widetilde{T_CX}
+S^2\widetilde{T_CX})\right.\\\notag
&\left.+512\widehat{A}(TX){\rm ch}(\wedge^4\widetilde{T_CX}+\wedge^2\widetilde{T_CX}\otimes\widetilde{T_CX}+\widetilde{T_CX}
\otimes\widetilde{T_CX})
+S^2\widetilde{T_CX}+\widetilde{T_CX}\right\}^{(16)}\\\notag
&=61920\left\{\widehat{A}(TX){\rm ch}(\triangle(X))+512\widehat{A}(TX)\right\}^{(16)}.
\end{align}
\end{thm}
\begin{cor}
Lex $X$ be an $16$-dimensional spin manifold without boundary, then
\begin{equation}
{\rm Ind}((D \otimes \triangle(X) \otimes \widetilde{T_CX})_+)\equiv 0 ~~{\rm mod} ~~16Z,
\end{equation}
\begin{equation}
 {\rm Ind}((D \otimes \triangle(X) \otimes (2\widetilde{T_CX}+\wedge^2\widetilde{T_CX}+\widetilde{T_CX}\otimes\widetilde{T_CX}
+S^2\widetilde{T_CX}))_+)\equiv 0 ~~{\rm mod} ~~32Z.
\end{equation}
\end{cor}

\begin{thm}
When ${\rm dim}X=20$, we have
\begin{align}
&\left\{\widehat{A}(TX){\rm ch}(\triangle(X)){\rm ch}(2\widetilde{T_CX})+2048\widehat{A}(TX){\rm ch}(\widetilde{T_CX}+\wedge^2\widetilde{T_CX})\right\}^{(20)}\\\notag
&=-264\left\{\widehat{A}(TX){\rm ch}(\triangle(X))+32\widehat{A}(TX)\right\}^{(20)},
\end{align}
\begin{align}
&\left\{\widehat{A}(TX){\rm ch}(\triangle(X)){\rm ch}(2\widetilde{T_CX}+\wedge^2\widetilde{T_CX}+\widetilde{T_CX}\otimes\widetilde{T_CX}
+S^2\widetilde{T_CX})\right.\\\notag
&\left.+2048\widehat{A}(TX){\rm ch}(\wedge^4\widetilde{T_CX}+\wedge^2\widetilde{T_CX}\otimes\widetilde{T_CX}+\widetilde{T_CX}
\otimes\widetilde{T_CX})
+S^2\widetilde{T_CX}+\widetilde{T_CX}\right\}^{(20)}\\\notag
&=-117288\left\{\widehat{A}(TX){\rm ch}(\triangle(X))+2048\widehat{A}(TX)\right\}^{(20)}.
\end{align}
\end{thm}
\begin{cor}
Lex $X$ be an $20$-dimensional spin manifold without boundary, then
\begin{equation}
{\rm Ind}((D \otimes \triangle(X) \otimes \widetilde{T_CX})_+)\equiv 0 ~~{\rm mod} ~~4Z,
\end{equation}
\begin{equation}
 {\rm Ind}((D \otimes \triangle(X) \otimes (2\widetilde{T_CX}+\wedge^2\widetilde{T_CX}+\widetilde{T_CX}\otimes\widetilde{T_CX}
+S^2\widetilde{T_CX}))_+)\equiv 0 ~~{\rm mod} ~~8Z.
\end{equation}
\end{cor}

Let $V$ be an $2m_0$ dimensional real Euclidean vector bundle with the Euclidean connection $\nabla^V$ and the curvature $R^V$.
Denote the first Pontryagin classes of $TX$ and $V$ by $p_1(X)$ and $p_1(V)$.

\begin{thm}
When ${\rm dim}X=8$ and $3p_1(V)=p_1(X)$ we have
\begin{align}
&240\left\{\widehat{A}(TX){\rm det}^{\frac{1}{2}}{\rm cosh}(\frac{\sqrt{-1}}{4 \pi}R^V)\right\}^{(8)}
=\left\{\widehat{A}(TX){\rm det}^{\frac{1}{2}}{\rm cosh}(\frac{\sqrt{-1}}{4 \pi}R^V)\right.\\
&\left.\cdot{\rm ch}(\widetilde{TX}+2\wedge^2\widetilde{V_C}-\widetilde{V_C}\otimes \widetilde{V_C}
+\widetilde{V_C})\right\}^{(8)}.\notag
\end{align}
\begin{align}
&2160\left\{\widehat{A}(TX){\rm det}^{\frac{1}{2}}{\rm cosh}(\frac{\sqrt{-1}}{4 \pi}R^V)\right\}^{(8)}
=\left\{\widehat{A}(TX){\rm det}^{\frac{1}{2}}{\rm cosh}(\frac{\sqrt{-1}}{4 \pi}R^V)\right.\\\notag
&\cdot{\rm ch}(S^2\widetilde{TX}+\widetilde{TX}+
(2\wedge^2\widetilde{V_C}-\widetilde{V_C}\otimes \widetilde{V_C}
+\widetilde{V_C})\otimes\widetilde{TX}
+\wedge^2\widetilde{V_C}\otimes\wedge^2\widetilde{V_C}\\\notag
&\left.+2\wedge^4\widetilde{V_C}-2\widetilde{V_C}\otimes \wedge^3\widetilde{V_C}+
2\widetilde{V_C}\otimes \wedge^2\widetilde{V_C}-\widetilde{V_C}\otimes \widetilde{V_C}\otimes \widetilde{V_C}
+\widetilde{V_C}+ \wedge^2\widetilde{V_C}
\right\}^{(8)}.\notag
\end{align}
\end{thm}
Let $X$ be closed oriented spinc-manifold and $L$ be the complex line bundle associated to the given spinc structure on $X$. We also
consider $L$ as a real vector bundle denoted by $L_R$. Denote by $c=c_1(L)$ the first Chern class of $L$. In Theorem 1.7, we set $V=L_R$ and
we have

\begin{cor}
When ${\rm dim}X=8$ and $3p_1(L_R)=p_1(X)$ we have
\begin{align}
&240\left\{\widehat{A}(TX){\rm exp}(\frac{c}{2})\right\}^{(8)}
=\left\{\widehat{A}(TX){\rm exp}(\frac{c}{2})\right.\\
&\left.\cdot{\rm ch}(\widetilde{TX}+2\wedge^2\widetilde{L_R\otimes C}-\widetilde{L_R\otimes C}\otimes \widetilde{L_R\otimes C}
+\widetilde{L_R\otimes C})\right\}^{(8)}.\notag
\end{align}
\begin{align}
&2160\left\{\widehat{A}(TX){\rm exp}(\frac{c}{2})\right\}^{(8)}
=\left\{\widehat{A}(TX){\rm exp}(\frac{c}{2}))\right.\\\notag
&\cdot{\rm ch}(S^2\widetilde{TX}+\widetilde{TX}+
(2\wedge^2\widetilde{L_R\otimes C}-\widetilde{L_R\otimes C}\otimes \widetilde{L_R\otimes C}\\\notag
&+\widetilde{L_R\otimes C})\otimes\widetilde{TX}
+\wedge^2\widetilde{L_R\otimes C}\otimes\wedge^2\widetilde{L_R\otimes C}\\\notag
&\left.+2\wedge^4\widetilde{L_R\otimes C}-2\widetilde{L_R\otimes C}\otimes \wedge^3\widetilde{L_R\otimes C}+
2\widetilde{L_R\otimes C}\otimes \wedge^2\widetilde{L_R\otimes C}\right.\\\notag
&\left.-\widetilde{L_R\otimes C}\otimes \widetilde{L_R\otimes C}\otimes \widetilde{L_R\otimes C}
+\widetilde{L_R\otimes C}+ \wedge^2\widetilde{L_R\otimes C}
\right\}^{(8)}.\notag
\end{align}
\end{cor}

\begin{cor}Lex $X$ be an $8$-dimensional spinc manifold without boundary and $3p_1(L_R)=p_1(X)$, then
\begin{equation}
{\rm Ind}((D^c \otimes (\widetilde{TX}+2\wedge^2\widetilde{L_R\otimes C}-\widetilde{L_R\otimes C}\otimes \widetilde{L_R\otimes C}
+\widetilde{L_R\otimes C}))_+)\equiv 0 ~~{\rm mod} ~~240Z.
\end{equation}
\begin{align}
&{\rm Ind}((D^c \otimes
(S^2\widetilde{TX}+\widetilde{TX}+
(2\wedge^2\widetilde{L_R\otimes C}-\widetilde{L_R\otimes C}\otimes \widetilde{L_R\otimes C}\\\notag
&+\widetilde{L_R\otimes C})\otimes\widetilde{TX}
+\wedge^2\widetilde{L_R\otimes C}\otimes\wedge^2\widetilde{L_R\otimes C}\\\notag
&+2\wedge^4\widetilde{L_R\otimes C}-2\widetilde{L_R\otimes C}\otimes \wedge^3\widetilde{L_R\otimes C}+
2\widetilde{L_R\otimes C}\otimes \wedge^2\widetilde{L_R\otimes C}\\\notag
&-\widetilde{L_R\otimes C}\otimes \widetilde{L_R\otimes C}\otimes \widetilde{L_R\otimes C}
+\widetilde{L_R\otimes C}+ \wedge^2\widetilde{L_R\otimes C})
_+)\equiv 0 ~~{\rm mod} ~~2160Z.\notag
\end{align}
\end{cor}

\begin{thm}
When ${\rm dim}X=12$ and $3p_1(V)=p_1(X)$ we have
\begin{align}
&-504\left\{\widehat{A}(TX){\rm det}^{\frac{1}{2}}{\rm cosh}(\frac{\sqrt{-1}}{4 \pi}R^V)\right\}^{(12)}
=\left\{\widehat{A}(TX){\rm det}^{\frac{1}{2}}{\rm cosh}(\frac{\sqrt{-1}}{4 \pi}R^V)\right.\\
&\left.\cdot{\rm ch}(\widetilde{TX}+2\wedge^2\widetilde{V_C}-\widetilde{V_C}\otimes \widetilde{V_C}
+\widetilde{V_C})\right\}^{(12)}.\notag
\end{align}
\begin{align}
&-16632\left\{\widehat{A}(TX){\rm det}^{\frac{1}{2}}{\rm cosh}(\frac{\sqrt{-1}}{4 \pi}R^V)\right\}^{(12)}
=\left\{\widehat{A}(TX){\rm det}^{\frac{1}{2}}{\rm cosh}(\frac{\sqrt{-1}}{4 \pi}R^V)\right.\\\notag
&\cdot{\rm ch}(S^2\widetilde{TX}+\widetilde{TX}+
(2\wedge^2\widetilde{V_C}-\widetilde{V_C}\otimes \widetilde{V_C}
+\widetilde{V_C})\otimes\widetilde{TX}
+\wedge^2\widetilde{V_C}\otimes\wedge^2\widetilde{V_C}\\\notag
&\left.+2\wedge^4\widetilde{V_C}-2\widetilde{V_C}\otimes \wedge^3\widetilde{V_C}+
2\widetilde{V_C}\otimes \wedge^2\widetilde{V_C}-\widetilde{V_C}\otimes \widetilde{V_C}\otimes \widetilde{V_C}
+\widetilde{V_C}+ \wedge^2\widetilde{V_C}
\right\}^{(12)}.\notag
\end{align}
\end{thm}

\begin{cor}
When ${\rm dim}X=12$ and $3p_1(L_R)=p_1(X)$ we have
\begin{align}
&-504\left\{\widehat{A}(TX){\rm exp}(\frac{c}{2})\right\}^{(12)}
=\left\{\widehat{A}(TX){\rm exp}(\frac{c}{2})\right.\\
&\left.\cdot{\rm ch}(\widetilde{TX}+2\wedge^2\widetilde{L_R\otimes C}-\widetilde{L_R\otimes C}\otimes \widetilde{L_R\otimes C}
+\widetilde{L_R\otimes C})\right\}^{(12)}.\notag
\end{align}
\begin{align}
&-16632\left\{\widehat{A}(TX){\rm exp}(\frac{c}{2})\right\}^{(12)}
=\left\{\widehat{A}(TX){\rm exp}(\frac{c}{2}))\right.\\\notag
&\cdot{\rm ch}(S^2\widetilde{TX}+\widetilde{TX}+
(2\wedge^2\widetilde{L_R\otimes C}-\widetilde{L_R\otimes C}\otimes \widetilde{L_R\otimes C}\\\notag
&+\widetilde{L_R\otimes C})\otimes\widetilde{TX}
+\wedge^2\widetilde{L_R\otimes C}\otimes\wedge^2\widetilde{L_R\otimes C}\\\notag
&\left.+2\wedge^4\widetilde{L_R\otimes C}-2\widetilde{L_R\otimes C}\otimes \wedge^3\widetilde{L_R\otimes C}+
2\widetilde{L_R\otimes C}\otimes \wedge^2\widetilde{L_R\otimes C}\right.\\\notag
&\left.-\widetilde{L_R\otimes C}\otimes \widetilde{L_R\otimes C}\otimes \widetilde{L_R\otimes C}
+\widetilde{L_R\otimes C}+ \wedge^2\widetilde{L_R\otimes C}
\right\}^{(12)}.\notag
\end{align}
\end{cor}

\begin{cor}
Lex $X$ be an $12$-dimensional spinc manifold without boundary and $3p_1(L_R)=p_1(X)$, then
\begin{equation}
{\rm Ind}((D^c \otimes (\widetilde{TX}+2\wedge^2\widetilde{L_R\otimes C}-\widetilde{L_R\otimes C}\otimes \widetilde{L_R\otimes C}
+\widetilde{L_R\otimes C}))_+)\equiv 0 ~~{\rm mod} ~~504Z.
\end{equation}
\begin{align}
&{\rm Ind}((D^c \otimes
(S^2\widetilde{TX}+\widetilde{TX}+
(2\wedge^2\widetilde{L_R\otimes C}-\widetilde{L_R\otimes C}\otimes \widetilde{L_R\otimes C}\\\notag
&+\widetilde{L_R\otimes C})\otimes\widetilde{TX}
+\wedge^2\widetilde{L_R\otimes C}\otimes\wedge^2\widetilde{L_R\otimes C}\\\notag
&+2\wedge^4\widetilde{L_R\otimes C}-2\widetilde{L_R\otimes C}\otimes \wedge^3\widetilde{L_R\otimes C}+
2\widetilde{L_R\otimes C}\otimes \wedge^2\widetilde{L_R\otimes C}\\\notag
&-\widetilde{L_R\otimes C}\otimes \widetilde{L_R\otimes C}\otimes \widetilde{L_R\otimes C}
+\widetilde{L_R\otimes C}+ \wedge^2\widetilde{L_R\otimes C})
_+)\equiv 0 ~~{\rm mod} ~~16632Z.\notag
\end{align}
\end{cor}

\begin{thm}
When ${\rm dim}X=16$ and $3p_1(V)=p_1(X)$ we have
\begin{align}
&480\left\{\widehat{A}(TX){\rm det}^{\frac{1}{2}}{\rm cosh}(\frac{\sqrt{-1}}{4 \pi}R^V)\right\}^{(12)}
=\left\{\widehat{A}(TX){\rm det}^{\frac{1}{2}}{\rm cosh}(\frac{\sqrt{-1}}{4 \pi}R^V)\right.\\
&\left.\cdot{\rm ch}(\widetilde{TX}+2\wedge^2\widetilde{V_C}-\widetilde{V_C}\otimes \widetilde{V_C}
+\widetilde{V_C})\right\}^{(12)}.\notag
\end{align}
\end{thm}

\begin{cor}
When ${\rm dim}X=16$ and $3p_1(L_R)=p_1(X)$ we have
\begin{align}
&480\left\{\widehat{A}(TX){\rm exp}(\frac{c}{2})\right\}^{(16)}
=\left\{\widehat{A}(TX){\rm exp}(\frac{c}{2})\right.\\
&\left.\cdot{\rm ch}(\widetilde{TX}+2\wedge^2\widetilde{L_R\otimes C}-\widetilde{L_R\otimes C}\otimes \widetilde{L_R\otimes C}
+\widetilde{L_R\otimes C})\right\}^{(16)}.\notag
\end{align}
\end{cor}

\begin{cor}
Lex $X$ be an $16$-dimensional spinc manifold without boundary and $3p_1(L_R)=p_1(X)$, then
\begin{equation}
{\rm Ind}((D^c \otimes (\widetilde{TX}+2\wedge^2\widetilde{L_R\otimes C}-\widetilde{L_R\otimes C}\otimes \widetilde{L_R\otimes C}
+\widetilde{L_R\otimes C}))_+)\equiv 0 ~~{\rm mod} ~~480Z.
\end{equation}
\end{cor}

\begin{thm}
When ${\rm dim}X=20$ and $3p_1(V)=p_1(X)$ we have
\begin{align}
&-264\left\{\widehat{A}(TX){\rm det}^{\frac{1}{2}}{\rm cosh}(\frac{\sqrt{-1}}{4 \pi}R^V)\right\}^{(20)}
=\left\{\widehat{A}(TX){\rm det}^{\frac{1}{2}}{\rm cosh}(\frac{\sqrt{-1}}{4 \pi}R^V)\right.\\
&\left.\cdot{\rm ch}(\widetilde{TX}+2\wedge^2\widetilde{V_C}-\widetilde{V_C}\otimes \widetilde{V_C}
+\widetilde{V_C})\right\}^{(20)}.\notag
\end{align}
\end{thm}

\begin{cor}
When ${\rm dim}X=20$ and $3p_1(L_R)=p_1(X)$ we have
\begin{align}
&-264\left\{\widehat{A}(TX){\rm exp}(\frac{c}{2})\right\}^{(20)}
=\left\{\widehat{A}(TX){\rm exp}(\frac{c}{2})\right.\\
&\left.\cdot{\rm ch}(\widetilde{TX}+2\wedge^2\widetilde{L_R\otimes C}-\widetilde{L_R\otimes C}\otimes \widetilde{L_R\otimes C}
+\widetilde{L_R\otimes C})\right\}^{(20)}.\notag
\end{align}
\end{cor}

\begin{cor}
Lex $X$ be an $20$-dimensional spinc manifold without boundary and $3p_1(L_R)=p_1(X)$, then
\begin{equation}
{\rm Ind}((D^c \otimes (\widetilde{TX}+2\wedge^2\widetilde{L_R\otimes C}-\widetilde{L_R\otimes C}\otimes \widetilde{L_R\otimes C}
+\widetilde{L_R\otimes C}))_+)\equiv 0 ~~{\rm mod} ~~264Z.
\end{equation}
\end{cor}

\begin{thm}
When ${\rm dim}X=10$ and $p_1(L_R)=p_1(X)$ we have
\begin{align}
&240\left\{\widehat{A}(TX){\rm exp}(\frac{c}{2})\right\}^{(10)}
=\left\{\widehat{A}(TX){\rm exp}(\frac{c}{2}){\rm ch}(\widetilde{TX}-\widetilde{L_R\otimes C})\right\}^{(10)}.\notag
\end{align}
\begin{align}
&2160\left\{\widehat{A}(TX){\rm exp}(\frac{c}{2})\right\}^{(10)}
=\left\{\widehat{A}(TX){\rm exp}(\frac{c}{2})\right.\\\notag
&\left.\cdot{\rm ch}(S^2\widetilde{TX}+\widetilde{TX}+
\wedge^2\widetilde{L_R\otimes C}-\widetilde{L_R\otimes C}
-\widetilde{TX}\otimes\widetilde{L_R\otimes C}
\right\}^{(10)}.\notag
\end{align}
\end{thm}

\begin{cor}
Lex $X$ be an $10$-dimensional spinc manifold without boundary and $p_1(L_R)=p_1(X)$, then
\begin{equation}
{\rm Ind}((D^c \otimes (\widetilde{TX}-\widetilde{L_R\otimes C}))_+)\equiv 0 ~~{\rm mod} ~~240Z.
\end{equation}
\begin{equation}
{\rm Ind}((D^c \otimes (S^2\widetilde{TX}+\widetilde{TX}+
\wedge^2\widetilde{L_R\otimes C}-\widetilde{L_R\otimes C}
-\widetilde{TX}\otimes\widetilde{L_R\otimes C}))_+)\equiv 0 ~~{\rm mod} ~~2160Z.
\end{equation}
\end{cor}

\begin{thm}
When ${\rm dim}X=14$ and $p_1(L_R)=p_1(X)$ we have
\begin{align}
&-504\left\{\widehat{A}(TX){\rm exp}(\frac{c}{2})\right\}^{(14)}
=\left\{\widehat{A}(TX){\rm exp}(\frac{c}{2}){\rm ch}(\widetilde{TX}-\widetilde{L_R\otimes C})\right\}^{(14)}.\notag
\end{align}
\begin{align}
&-16632\left\{\widehat{A}(TX){\rm exp}(\frac{c}{2})\right\}^{(14)}
=\left\{\widehat{A}(TX){\rm exp}(\frac{c}{2})\right.\\\notag
&\left.\cdot{\rm ch}(S^2\widetilde{TX}+\widetilde{TX}+
\wedge^2\widetilde{L_R\otimes C}-\widetilde{L_R\otimes C}
-\widetilde{TX}\otimes\widetilde{L_R\otimes C}
\right\}^{(14)}.\notag
\end{align}
\end{thm}

\begin{cor}
Lex $X$ be an $14$-dimensional spinc manifold without boundary and $p_1(L_R)=p_1(X)$, then
\begin{equation}
{\rm Ind}((D^c \otimes (\widetilde{TX}-\widetilde{L_R\otimes C}))_+)\equiv 0 ~~{\rm mod} ~~504Z.
\end{equation}
\begin{equation}
{\rm Ind}((D^c \otimes (S^2\widetilde{TX}+\widetilde{TX}+
\wedge^2\widetilde{L_R\otimes C}-\widetilde{L_R\otimes C}
-\widetilde{TX}\otimes\widetilde{L_R\otimes C}))_+)\equiv 0 ~~{\rm mod} ~~16632Z.
\end{equation}
\end{cor}

\begin{thm}
When ${\rm dim}X=18$ and $p_1(L_R)=p_1(X)$ we have
\begin{align}
&480\left\{\widehat{A}(TX){\rm exp}(\frac{c}{2})\right\}^{(18)}
=\left\{\widehat{A}(TX){\rm exp}(\frac{c}{2}){\rm ch}(\widetilde{TX}-\widetilde{L_R\otimes C})\right\}^{(18)}.\notag
\end{align}
\end{thm}

\begin{cor}
Lex $X$ be an $18$-dimensional spinc manifold without boundary and $p_1(L_R)=p_1(X)$, then
\begin{equation}
{\rm Ind}((D^c \otimes (\widetilde{TX}-\widetilde{L_R\otimes C}))_+)\equiv 0 ~~{\rm mod} ~~480Z.
\end{equation}
\end{cor}
\begin{thm}
When ${\rm dim}X=22$ and $p_1(L_R)=p_1(X)$ we have
\begin{align}
&-264\left\{\widehat{A}(TX){\rm exp}(\frac{c}{2})\right\}^{(22)}
=\left\{\widehat{A}(TX){\rm exp}(\frac{c}{2}){\rm ch}(\widetilde{TX}-\widetilde{L_R\otimes C})\right\}^{(22)}.\notag
\end{align}
\end{thm}
\begin{cor}
Lex $X$ be an $22$-dimensional spinc manifold without boundary and $p_1(L_R)=p_1(X)$, then
\begin{equation}
{\rm Ind}((D^c \otimes (\widetilde{TX}-\widetilde{L_R\otimes C}))_+)\equiv 0 ~~{\rm mod} ~~264Z.
\end{equation}
\end{cor}

\vskip 1 true cm

\section{The proof of Theorems in Section 1}

\indent Let $X$ be a $4k$-dimensional spin manifold and $\triangle(X)$ be the spinor bundle.
 Set
 \begin{equation}
   \Theta_1(T_{C}X)=
   \bigotimes _{n=1}^{\infty}S_{q^n}(\widetilde{T_CX})\otimes
\bigotimes _{m=1}^{\infty}\wedge_{q^m}(\widetilde{T_CX})
,\end{equation}
\begin{equation}
\Theta_2(T_{C}X)=\bigotimes _{n=1}^{\infty}S_{q^n}(\widetilde{T_CX})\otimes
\bigotimes _{m=1}^{\infty}\wedge_{-q^{m-\frac{1}{2}}}(\widetilde{T_CX}),
\end{equation}
\begin{equation}
\Theta_3(T_{C}X)=\bigotimes _{n=1}^{\infty}S_{q^n}(\widetilde{T_CX})\otimes
\bigotimes _{m=1}^{\infty}\wedge_{q^{m-\frac{1}{2}}}(\widetilde{T_CX})
.\end{equation}
Let
\begin{align}
Q(X,\tau)&={\rm Ind}((D\otimes[\triangle(X)\otimes \Theta_1(T_{C}X)+2^{2k}\Theta_2(T_{C}X)+2^{2k}\Theta_3(T_{C}X)])_+)\\\notag
&=\int_X\widehat{A}(TX)[{\rm ch}(\triangle(X)){\rm ch}(\Theta_1(T_{C}X))+2^{2k}{\rm ch}(\Theta_2(T_{C}X))+2^{2k}{\rm ch}(\Theta_3(T_{C}X))].
\end{align}
   \indent We first recall the four Jacobi theta functions are
   defined as follows( cf. \cite{Ch}):
 \begin{equation}  \theta(v,\tau)=2q^{\frac{1}{8}}{\rm sin}(\pi
   v)\prod_{j=1}^{\infty}[(1-q^j)(1-e^{2\pi\sqrt{-1}v}q^j)(1-e^{-2\pi\sqrt{-1}v}q^j)],
   \end{equation}
\begin{equation}\theta_1(v,\tau)=2q^{\frac{1}{8}}{\rm cos}(\pi
   v)\prod_{j=1}^{\infty}[(1-q^j)(1+e^{2\pi\sqrt{-1}v}q^j)(1+e^{-2\pi\sqrt{-1}v}q^j)],\end{equation}
\begin{equation}\theta_2(v,\tau)=\prod_{j=1}^{\infty}[(1-q^j)(1-e^{2\pi\sqrt{-1}v}q^{j-\frac{1}{2}})
(1-e^{-2\pi\sqrt{-1}v}q^{j-\frac{1}{2}})],\end{equation}
\begin{equation}\theta_3(v,\tau)=\prod_{j=1}^{\infty}[(1-q^j)(1+e^{2\pi\sqrt{-1}v}q^{j-\frac{1}{2}})
(1+e^{-2\pi\sqrt{-1}v}q^{j-\frac{1}{2}})],\end{equation}
\noindent
where $q=e^{2\pi\sqrt{-1}\tau}$ with $\tau\in\textbf{H}$, the upper
half complex plane. Let
\begin{equation}\theta'(0,\tau)=\frac{\partial\theta(v,\tau)}{\partial
v}|_{v=0}.\end{equation} \noindent Then the following Jacobi identity
(cf. \cite{Ch} ) holds,
\begin{equation}\theta'(0,\tau)=\pi\theta_1(0,\tau)\theta_2(0,\tau)\theta_3(0,\tau).\end{equation}
\noindent Denote $SL_2({\bf Z})=\left\{\left(\begin{array}{cc}
\ a & b  \\
 c  & d
\end{array}\right)\mid a,b,c,d \in {\bf Z},~ad-bc=1\right\}$ the
modular group. Let $S=\left(\begin{array}{cc}
\ 0 & -1  \\
 1  & 0
\end{array}\right),~T=\left(\begin{array}{cc}
\ 1 &  1 \\
 0  & 1
\end{array}\right)$ be the two generators of $SL_2(\bf{Z})$. They
act on $\textbf{H}$ by $S\tau=-\frac{1}{\tau},~T\tau=\tau+1$. One
has the following transformation laws of theta functions under the
actions of $S$ and $T$ (cf. \cite{Ch} ):
\begin{equation}\theta(v,\tau+1)=e^{\frac{\pi\sqrt{-1}}{4}}\theta(v,\tau),~~\theta(v,-\frac{1}{\tau})
=\frac{1}{\sqrt{-1}}\left(\frac{\tau}{\sqrt{-1}}\right)^{\frac{1}{2}}e^{\pi\sqrt{-1}\tau
v^2}\theta(\tau v,\tau);\end{equation}
\begin{equation}\theta_1(v,\tau+1)=e^{\frac{\pi\sqrt{-1}}{4}}\theta_1(v,\tau),~~\theta_1(v,-\frac{1}{\tau})
=\left(\frac{\tau}{\sqrt{-1}}\right)^{\frac{1}{2}}e^{\pi\sqrt{-1}\tau
v^2}\theta_2(\tau v,\tau);\end{equation}
\begin{equation}\theta_2(v,\tau+1)=\theta_3(v,\tau),~~\theta_2(v,-\frac{1}{\tau})
=\left(\frac{\tau}{\sqrt{-1}}\right)^{\frac{1}{2}}e^{\pi\sqrt{-1}\tau
v^2}\theta_1(\tau v,\tau);\end{equation}
\begin{equation}\theta_3(v,\tau+1)=\theta_2(v,\tau),~~\theta_3(v,-\frac{1}{\tau})
=\left(\frac{\tau}{\sqrt{-1}}\right)^{\frac{1}{2}}e^{\pi\sqrt{-1}\tau
v^2}\theta_3(\tau v,\tau),\end{equation}
 \begin{equation}\theta'(v,\tau+1)=e^{\frac{\pi\sqrt{-1}}{4}}\theta'(v,\tau),~~
 \theta'(0,-\frac{1}{\tau})=\frac{1}{\sqrt{-1}}\left(\frac{\tau}{\sqrt{-1}}\right)^{\frac{1}{2}}
\tau\theta'(0,\tau).\end{equation}
\begin{defn} A modular form over $\Gamma$, a
 subgroup of $SL_2({\bf Z})$, is a holomorphic function $f(\tau)$ on
 $\textbf{H}$ such that
\begin{equation} f(g\tau):=f\left(\frac{a\tau+b}{c\tau+d}\right)=\chi(g)(c\tau+d)^kf(\tau),
 ~~\forall g=\left(\begin{array}{cc}
\ a & b  \\
 c & d
\end{array}\right)\in\Gamma,\end{equation}
\noindent where $\chi:\Gamma\rightarrow {\bf C}^{\star}$ is a
character of $\Gamma$. $k$ is called the weight of $f$.
\end{defn}
\begin{thm}(\cite{Li2})
$Q(X,\tau)$ is a modular form over $SL_2({\bf Z})$ with the weight $2k$.
\end{thm}
\begin{proof} Let $\{\pm 2\pi\sqrt{-1}x_j\}~(1\leq j \leq 2k)$ be the formal Chern roots for $T_CX$, then we have
 \begin{equation}{\widehat{A}(TX,\nabla^{TX})}{\rm ch}(\triangle(X)){\rm ch}(\Theta_1(T_{C}X))=\prod_{j=1}^{2k}\frac{2 x_j\theta'(0,\tau)}{\theta(x_j,\tau)}\frac{\theta_1(x_j,\tau)}{\theta_1(0,\tau)},\end{equation}
\begin{equation}{\widehat{A}(TX,\nabla^{TX})}{\rm ch}(2^k\Theta_2(T_{C}X))=\prod_{j=1}^{2k}\frac{2 x_j\theta'(0,\tau)}{\theta(x_j,\tau)}\frac{\theta_2(x_j,\tau)}{\theta_2(0,\tau)},\end{equation}
 \begin{equation}{\widehat{A}(TX,\nabla^{TX})}{\rm ch}(2^k\Theta_3(T_{C}X))=\prod_{j=1}^{2k}\frac{2 x_j\theta'(0,\tau)}{\theta(x_j,\tau)}\frac{\theta_3(x_j,\tau)}{\theta_3(0,\tau)},\end{equation}
 So we have
\begin{equation}Q(X,\tau)=\prod_{j=1}^{2k}\left(\frac{2x_j\theta'(0,\tau)}{\theta(x_j,\tau)}
\left(\prod_{j=1}^{2k}\frac{\theta_1(x_j,\tau)}{\theta_1(0,\tau)}+\prod_{j=1}^{2k}\frac{\theta_2(x_j,\tau)}
{\theta_2(0,\tau)}+\prod_{j=1}^{2k}\frac{\theta_3(x_j,\tau)}{\theta_3(0,\tau)}
\right)\right).\end{equation}
By (2.11)-(2.15), we have $Q(X,\tau+1)=Q(X,\tau)$ and $Q(X,-\frac{1}{\tau})=\tau^{2k}Q(X,\tau)$, so
$Q(X,\tau)$ is a modular form over $SL_2({\bf Z})$ with the weight $2k$.
\end{proof}
{\bf The proof of Theorem 1.1:}
It is well known that modular forms over $SL_2({\bf Z})$ can be expressed as polynomials of the Einsentein series $E_4(\tau)$ and $E_6(\tau)$,
where
 \begin{equation}
E_4(\tau)=1+240q+2160q^2+6720q^3+\cdots,
\end{equation}
\begin{equation}
E_6(\tau)=1-504q-16632q^2-122976q^3+\cdots.
\end{equation}
Their weights are $4$ and $6$ respectively. When ${\rm dim}X=8$, then $Q(X,\tau)$ is a modular form over $SL_2({\bf Z})$ with the weight $4$. $Q(X,\tau)$ must be a multiple of
\begin{equation}
Q(X,\tau)=\lambda E_4(\tau).
\end{equation}
By (2.1)-(2.3), we have
\begin{equation}
\Theta_1(T_{C}X)=1+2q\widetilde{T_CX}+O(q^2),
\end{equation}
\begin{equation}
\Theta_2(T_{C}X)=1-q^{\frac{1}{2}}\widetilde{T_CX}+q(\widetilde{T_CX}+\wedge^2\widetilde{T_CX})+O(q^{\frac{3}{2}}),
\end{equation}
\begin{equation}
\Theta_3(T_{C}X)=1+q^{\frac{1}{2}}\widetilde{T_CX}+q(\widetilde{T_CX}+\wedge^2\widetilde{T_CX})+O(q^{\frac{3}{2}}),
\end{equation}
so
\begin{align}
&Q(X,\tau)=\left[\widehat{A}(TX){\rm ch}(\triangle(X))+32\widehat{A}(TX)\right]^{(8)}\\\notag
&+\left[2\widehat{A}(TX){\rm ch}(\triangle(X)){\rm ch}(\widetilde{T_CX})+32\widehat{A}(TX){\rm ch} (\widetilde{T_CX}+\wedge^2\widetilde{T_CX})\right]^{(8)}q\\\notag
&+\left[\widehat{A}(TX){\rm ch}(\triangle(X))
{\rm ch}(2\widetilde{T_CX}+\wedge^2\widetilde{T_CX}+\widetilde{T_CX}\otimes \widetilde{T_CX}+S^2\widetilde{T_CX})\right.\\\notag
&\left.+32\widehat{A}(TX){\rm ch}(\wedge^4\widetilde{T_CX}+\wedge^2\widetilde{T_CX}\otimes \widetilde{T_CX}
+\widetilde{T_CX}\otimes \widetilde{T_CX}+S^2\widetilde{T_CX}+
\widetilde{T_CX})\right]^{(8)}q^2+\cdots.
\end{align}
By (2.21), (2.23) and (2.27), we get Theorem 1.1. $\Box$

{\bf The proof of Theorem 1.3:}
When ${\rm dim}X=12$, then $Q(X,\tau)$ is a modular form over $SL_2({\bf Z})$ with the weight $6$, so
\begin{equation}Q(X,\tau)=\lambda E_6(\tau),
\end{equation}
where $\lambda$ are degree $6$ forms.
When ${\rm dim}X=12$, direct computations show that
\begin{align}
&Q(X,\tau)=\left[\widehat{A}(TX){\rm ch}(\triangle(X))+128\widehat{A}(TX)\right]^{(12)}\\\notag
&+\left[2\widehat{A}(TX){\rm ch}(\triangle(X)){\rm ch}(\widetilde{T_CX})+128\widehat{A}(TX){\rm ch} (\widetilde{T_CX}+\wedge^2\widetilde{T_CX})\right]^{(12)}q\\\notag
&+\left[\widehat{A}(TX){\rm ch}(\triangle(X))
{\rm ch}(2\widetilde{T_CX}+\wedge^2\widetilde{T_CX}+\widetilde{T_CX}\otimes \widetilde{T_CX}+S^2\widetilde{T_CX})\right.\\\notag
&\left.+128\widehat{A}(TX){\rm ch}(\wedge^4\widetilde{T_CX}+\wedge^2\widetilde{T_CX}\otimes \widetilde{T_CX}
+\widetilde{T_CX}\otimes \widetilde{T_CX}+S^2\widetilde{T_CX}+
\widetilde{T_CX})\right]^{(12)}q^2+\cdots.
\end{align}
In (2.28), we compare the coefficients of $1$, we get three equations about $\lambda$. By
(2.28) and (2.29) we get Theorem 1.3. $\Box$

{\bf The proof of Theorem 1.5:}
When ${\rm dim}X=16$, then $Q(X,\tau)$ is a modular form over $SL_2({\bf Z})$ with the weight $8$, so
\begin{equation}Q(X,\tau)=\lambda E_4(\tau)^2,
\end{equation}
where $\lambda$ is degree $8$ forms. By (2.21), we have
\begin{equation}
E_4(\tau)^2=1+480q+61920q^2+\cdots.
\end{equation}
When ${\rm dim}X=16$, direct computations show that

\begin{align}
&Q(X,\tau)=\left[\widehat{A}(TX){\rm ch}(\triangle(X))+512\widehat{A}(TX)\right]^{(16)}\\\notag
&+\left[2\widehat{A}(TX){\rm ch}(\triangle(X)){\rm ch}(\widetilde{T_CX})+512\widehat{A}(TX){\rm ch} (\widetilde{T_CX}+\wedge^2\widetilde{T_CX})\right]^{(16)}q\\\notag
&+\left[\widehat{A}(TX){\rm ch}(\triangle(X))
{\rm ch}(2\widetilde{T_CX}+\wedge^2\widetilde{T_CX}+\widetilde{T_CX}\otimes \widetilde{T_CX}+S^2\widetilde{T_CX})\right.\\\notag
&\left.+512\widehat{A}(TX){\rm ch}(\wedge^4\widetilde{T_CX}+\wedge^2\widetilde{T_CX}\otimes \widetilde{T_CX}
+\widetilde{T_CX}\otimes \widetilde{T_CX}+S^2\widetilde{T_CX}+
\widetilde{T_CX})\right]^{(16)}q^2+\cdots.
\end{align}
By (2.30)-(2.32), we get Theorem 1.5.$\Box$

{\bf The proof of Theorem 1.7:}
When ${\rm dim}X=20$, then $Q(X,\tau)$ is a modular form over $SL_2({\bf Z})$ with the weight $10$, so
\begin{equation}Q(X,\tau)=\lambda E_4(\tau)E_6(\tau),
\end{equation}
where $\lambda$ is degree $10$ forms. By (2.21) and (2.22), we have
\begin{equation}
E_4(\tau)E_6(\tau)=1-264q-117288q^2+\cdots.
\end{equation}
When ${\rm dim}X=20$, direct computations show that

\begin{align}
&Q(X,\tau)=\left[\widehat{A}(TX){\rm ch}(\triangle(X))+2048\widehat{A}(TX)\right]^{(20)}\\\notag
&+\left[2\widehat{A}(TX){\rm ch}(\triangle(X)){\rm ch}(\widetilde{T_CX})+2048\widehat{A}(TX){\rm ch} (\widetilde{T_CX}+\wedge^2\widetilde{T_CX})\right]^{(20)}q\\\notag
&+\left[\widehat{A}(TX){\rm ch}(\triangle(X))
{\rm ch}(2\widetilde{T_CX}+\wedge^2\widetilde{T_CX}+\widetilde{T_CX}\otimes \widetilde{T_CX}+S^2\widetilde{T_CX})\right.\\\notag
&\left.+2048\widehat{A}(TX){\rm ch}(\wedge^4\widetilde{T_CX}+\wedge^2\widetilde{T_CX}\otimes \widetilde{T_CX}
+\widetilde{T_CX}\otimes \widetilde{T_CX}+S^2\widetilde{T_CX}+
\widetilde{T_CX})\right]^{(20)}q^2+\cdots.
\end{align}
By (2.33)-(2.35), we get Theorem 1.7.$\Box$

Let ${\rm dim}X=4k$. Let $V$ be an $2m_0$ dimensional real Euclidean vector bundle with the Euclidean connection $\nabla^V$ and the curvature $R^V$.  Let $\{\pm 2\pi\sqrt{-1}u_r\}~(1\leq r \leq m_0)$ be the formal Chern roots for $V\otimes C$.
Let
\begin{align}
&Q(X,V,\tau)=\left\{\widehat{A}(TX){\rm det}^{\frac{1}{2}}{\rm cosh}(\frac{\sqrt{-1}}{4 \pi}R^V){\rm ch}\left[\bigotimes _{n=1}^{\infty}S_{q^n}(\widetilde{T_CM})\right.\right.\\\notag
&\left.\left.\otimes
\bigotimes _{m=1}^{\infty}\wedge_{q^m}(\widetilde{V_C})
\otimes \bigotimes _{r=1}^{\infty}\wedge
_{q^{r-\frac{1}{2}}}(\widetilde{V_C}))\otimes \bigotimes
_{s=1}^{\infty}\wedge _{-q^{s-\frac{1}{2}}}(\widetilde{V_C})\right]\right\}^{(4k)}.
\end{align}
Then
\begin{equation}Q(X,V,\tau)=\left(\prod_{j=1}^{2k}\frac{x_j\theta'(0,\tau)}{\theta(x_j,\tau)}
\prod_{r=1}^{m_0}\left(\frac{\theta_1(u_r,\tau)}{\theta_1(0,\tau)}\frac{\theta_2(u_r,\tau)}{\theta_2(0,\tau)}
\frac{\theta_3(u_r,\tau)}{\theta_3(0,\tau)}
\right)\right)^{(4k)}.\end{equation}
By (2.11)-(2.15), we have $Q(X,V,\tau+1)=Q(X,V,\tau)$ and $Q(X,V-\frac{1}{\tau})=\tau^{2k}Q(X,V,\tau)$ if $p_1(M)=3p_1(V)$, so
\begin{thm}
Let ${\rm dim}X=4k$. If $p_1(M)=3p_1(V)$, then $Q(X,V,\tau)$ is a modular form over $SL_2({\bf Z})$ with the weight $2k$.
\end{thm}
Direct computations show
\begin{align}
Q(X,V,\tau)=
&\left\{\widehat{A}(TX){\rm det}^{\frac{1}{2}}{\rm cosh}(\frac{\sqrt{-1}}{4 \pi}R^V)\right\}^{(4k)}
+q\left\{\widehat{A}(TX){\rm det}^{\frac{1}{2}}{\rm cosh}(\frac{\sqrt{-1}}{4 \pi}R^V)\right.\\\notag
&\left.\cdot{\rm ch}(\widetilde{TX}+2\wedge^2\widetilde{V_C}-\widetilde{V_C}\otimes \widetilde{V_C}
+\widetilde{V_C})\right\}^{(4k)}+q^2\left\{\widehat{A}(TX){\rm det}^{\frac{1}{2}}{\rm cosh}(\frac{\sqrt{-1}}{4 \pi}R^V)\right.\\\notag
&\cdot{\rm ch}(S^2\widetilde{TX}+\widetilde{TX}+
(2\wedge^2\widetilde{V_C}-\widetilde{V_C}\otimes \widetilde{V_C}
+\widetilde{V_C})\otimes\widetilde{TX}
+\wedge^2\widetilde{V_C}\otimes\wedge^2\widetilde{V_C}\\\notag
&\left.+2\wedge^4\widetilde{V_C}-2\widetilde{V_C}\otimes \wedge^3\widetilde{V_C}+
2\widetilde{V_C}\otimes \wedge^2\widetilde{V_C}-\widetilde{V_C}\otimes \widetilde{V_C}\otimes \widetilde{V_C}
+\widetilde{V_C}+ \wedge^2\widetilde{V_C}
\right\}^{(4k)}+\cdots.\notag
\end{align}
When $dim X=8$, then $Q(X,V,\tau)$ is a modular form over $SL_2({\bf Z})$ with the weight $4$ and
$Q(X,V,\tau)=\lambda E_4(\tau)$. We get Theorem 1.9. When $dim X=12$, then $Q(X,V,\tau)$ is a modular form over $SL_2({\bf Z})$ with the weight $6$ and
$Q(X,V,\tau)=\lambda E_6(\tau)$. We get Theorem 1.12. When $dim X=16$, then $Q(X,V,\tau)$ is a modular form over $SL_2({\bf Z})$ with the weight $8$ and
$Q(X,V,\tau)=\lambda E_4(\tau)^2$. We get Theorem 1.15. When $dim X=20$, then $Q(X,V,\tau)$ is a modular form over $SL_2({\bf Z})$ with the weight $10$ and
$Q(X,V,\tau)=\lambda E_4(\tau)E_6(\tau)$ and $E_4(\tau)E_6(\tau)=1-264q+\cdots$. We get Theorem 1.18.\\

Let $X$ be closed oriented spinc-manifold and $L$ be the complex line bundle associated to the given spinc structure on $X$. We also
consider $L$ as a real vector bundle denoted by $L_R$. Denote by $c=c_1(L)$ the first Chern class of $L$. Let ${\rm dim}X=4k+2$ and $u=-\frac{\sqrt{-1}}{2 \pi}c$.
Let
\begin{align}
&Q(X,L,\tau)=\left\{\widehat{A}(TX){\rm exp}(\frac{c}{2}){\rm ch}\left[\bigotimes _{n=1}^{\infty}S_{q^n}(\widetilde{T_CM})
\otimes
\bigotimes _{m=1}^{\infty}\wedge_{-q^m}(\widetilde{V_C})\right]
\right\}^{(4k+2)}.
\end{align}
Then
\begin{equation}Q(X,L,\tau)=\left\{\left(\prod_{j=1}^{2k+1}\frac{x_j\theta'(0,\tau)}{\theta(x_j,\tau)}\right)
\frac{\sqrt{-1}\theta(u,\tau)}{\theta_1(0,\tau)\theta_2(0,\tau)
\theta_3(0,\tau)}
\right\}^{(4k+2)}.
\end{equation}
\begin{thm}(\cite{CHZ})
Let ${\rm dim}X=4k+2$. If $p_1(M)=p_1(L_R)$, then $Q(X,L,\tau)$ is a modular form over $SL_2({\bf Z})$ with the weight $2k$.
\end{thm}
By Theorem 2.4, similar to Theorems 1.9-1.20, we can get Theorems 1.21-1.28.\\

\section{Acknowledgements}

 The author was supported in part by NSFC No.11771070. The author is indebted to Prof. F. Han for helpful comments. The author also thank the referee for his (or her) careful reading and helpful comments.

\section{Data availability}

No data was gathered for this article.

\section{Conflict of interest}

The authors have no relevant financial or non-financial interests to disclose.

\vskip 1 true cm


\bigskip
\bigskip
\indent{Y. Wang}\\
 \indent{School of Mathematics and Statistics,
Northeast Normal University, Changchun Jilin, 130024, China }\\
\indent E-mail: {\it wangy581@nenu.edu.cn }\\
\indent{J. Guan}\\
 \indent{School of Mathematics and Statistics,
Northeast Normal University, Changchun Jilin, 130024, China }\\
\indent E-mail: {\it guanjy@nenu.edu.cn }\\


\begin{thebibliography}{20}

\bibitem{AW}
 L. Alvarez-Gaum\'{e}, E. Witten, Graviational
anomalies,
Nucl. Phys. B234 (1983), 269-330.

\bibitem{Ch} K. Chandrasekharan, {\it Elliptic
Functions}, Spinger-Verlag, 1985.

\bibitem{CH}Q. Chen, F. Han, Modular invariance and twisted
anomaly cancellations of characteristic numbers, Trans. Amer.
Math. Soc. 361 (2009), 1463-1493

\bibitem{CHZ}
Q. Chen, F. Han, W. Zhang, Generalized Witten genus and vanishing theorems. J. Differential Geom. 88(2011), no. 1, 1-40.

\bibitem{HHLZ}F. Han, R. Huang, K. Liu and W. Zhang, Cubic forms, anomaly cancellation and modularity. Adv. Math. 394(2022), Paper No. 108023,46pp.

\bibitem{HLZ1}F. Han, K. Liu and W. Zhang, Modular forms and generalized anomaly cancellation
formulas, J. Geom. Phys, 62 (2012) 1038-1053.

\bibitem{HLZ2}F. Han, K. Liu and W. Zhang, Anomaly cancellation and modularity, II: the $E_8\times E_8$
case. Sci. China Math. 60(2017), no. 6, 985-994.

\bibitem{HZ1} F. Han, W. Zhang, ${\rm
Spin}^c$-manifold and elliptic genera,
C. R. Acad. Sci. Paris Serie I., 336 (2003), 1011-1014.

\bibitem{HZ2}F. Han, W. Zhang, Modular invariance, characteristic
numbers and eta Invariants, J.
Diff. Geom. 67 (2004), 257-288.

\bibitem{Li1}K. Liu, Modular invariance and characteristic
numbers. Commu. Math. Phys. 174 (1995), 29-42.

\bibitem{Li2}K. Liu, On elliptic genera and theta-functions.
Topology 35 (1996), no. 3, 617--640.

\bibitem{Zh} W. Zhang, {\it Lectures on Chern-weil Theory
and Witten Deformations.} Nankai Tracks in Mathematics Vol. 4, World
Scientific, Singapore, 2001.


\end{thebibliography}
\end{document}